\documentclass{amsart}

\usepackage{amssymb,amscd,amsthm,latexsym, amsmath}
\usepackage[all]{xy}

\newtheorem{Lemma}{Lemma}[section]
\newtheorem{remark}[Lemma]{Remark}
\newtheorem{remarks}[Lemma]{Remarks}
\newtheorem{theorem}[Lemma]{Theorem}
\newtheorem{lemma}[Lemma]{Lemma}
\newtheorem{proposition}[Lemma]{Proposition}

\newtheorem{example}[Lemma]{Example}
\newtheorem{examples}[Lemma]{Examples}

\newdir{ >}{{}*!/-8pt/@{>}}
\def\R{ \mathbb{R} }
\def\C{ \mathbb{C} }
\def\Z{ \mathbb{Z} }

\def\OO{ \mathbb{O} }

\def\NN{ \mathbb{N} }

\newcommand{\Grp}{\mathsf{Grp}}
\newcommand{\lSdp}{\mbox{\rm -}{\mathsf{Sdp}}}

    \newcommand{\Hom}{\operatorname{Hom}}
\newcommand{\Tr}{\operatorname{Tr}}

\def\cleft{\hbox{[\kern-.16em\hbox{[}}}
\def\cright{\hbox{]\kern-.16em\hbox{]}}}

\newcommand{\Rng}{\mathsf{Rng}}

\newcommand{\Aut}{\mathsf{Aut}}

\newcommand{\lMod}{\mbox{\rm -}{\mathsf{Mod}}}
\newcommand{\lAlg}{\mbox{\rm -}{\mathsf{Alg}}}
\newcommand{\BFlAlg}{\mbox{\rm -}{\mathsf{BFAlg}}}

\DeclareMathOperator{\Imm}{Im}

\newcommand{\id}{\mbox{\rm id}}
\newcommand{\Ring}{\mathsf{Ring}}

\begin{document}
		\title{Algebras with a bilinear form, and Idempotent endomorphisms}

\author{Alberto Facchini}
	\address{Dipartimento di Matematica ``Tullio Levi-Civita'', Universit\`a di 
Padova, 35121 Padova, Italy}
 \email{facchini@math.unipd.it}
\thanks{The first author is partially supported by Ministero dell'Istruzione, dell'Universit\`a e della Ricerca (Progetto di ricerca di rilevante interesse nazionale ``Categories, Algebras: Ring-Theoretical and Homological Approaches (CARTHA)''), Fondazione Cariverona (Research project ``Reducing complexity in algebra, logic, combinatorics - REDCOM'' within the framework of the programme Ricerca Scientifica di Eccellenza 2018), and the Department of Mathematics ``Tullio Levi-Civita'' of the University of Padua (Research programme DOR1828909 ``Anelli e categorie di moduli'').}
\author{Leila Heidari Zadeh}
\address{Department of Mathematics and Statistics, Shoushtar Branch, Islamic Azad University,  Shoushtar, Iran.}
 \email{heidaryzadehleila@yahoo.com}
   \begin{abstract} The category  $k\BFlAlg$ of all  $k$-algebras with a bilinear form, whose objects are all pairs $(R,b)$ where $R$ is a $k$-algebra and $b\colon R\times R\to k$ is a bilinear mapping, is equivalent to the category of unital $k$-algebras $A$ for which the canonical homomorphism $(k,1)\to(A,1_A)$ of unital $k$-algebras is a splitting monomorphism in $k \lMod$. Call the left inverses of this splitting monomorphism  ``weak augmentations" of the algebra. There is a category isomorphism between the category of $k$-algebras with a weak augmentation and the category of unital $k$-algebras $(A,b_A)$ with a bilinear form $b_A$ compatible with the multiplication of $A$, i.e., such that $b_A(x,y)=b_A(z,w)$ for all $x,y,z,w\in A$ for which $xy=zw$.
\end{abstract}

    \maketitle

	\date{}
	
	\maketitle
	
	\section*{Introduction}
	
Let $k$ be a commutative ring with identity $1$ and $A$ a not-necessarily associative $k$-algebra. If the $k$-algebra  $A$ has an identity $1_A$, there is a canonical homomorphism $(k,1)\to(A,1_A)$ of unital $k$-algebras.
Suppose that this canonical homomorphism is a splitting monomorphism in the category of unital $k$-algebras. Then $A=I\oplus k$, where $I$ is a two-sided ideal of $A$. Writing the elements of $I\oplus k$ as pairs, we have that $(i,\lambda)(i,\lambda')=(\lambda i'+\lambda' i+ii', \lambda\lambda')$, i.e., we have something very similar to what is sometimes called the  ``Dorroh extension" of $I$, that is, adjoining an identity to the $k$-algebra $I$. The left inverse of the splitting monomorphism $(k,1)\to(A,1_A)$ is an {\em augmentation} for $A$. It is well known that this Dorroh extension yields a category equivalence between the category of $k$-algebras and the category of unital $k$-algebras with an augmentation. The direct-sum decompositions $A=I\oplus k$ or, better, the splitting monomorphisms $(k,1)\to(A,1_A)$, correspond to idempotent endomorphisms of $A$ in the category of unital $k$-algebras having image $k\cdot 1_A$. (For a more detailed explanation and the notations, see Section~\ref{jbop}). 

But very often the $k$-algebra monomorphism $(k,1)\to(A,1_A)$ does not split in the category $k \lAlg_1$ of unital $k$-algebras, but only in the category $k \lMod$ of all $k$-modules. In this case, we only have that $A=S\oplus k$ as $k$-modules, where $S$ is a $k$-submodule of $A$. Denoting by $\pi_S\colon A\to S$ and $\pi_k\colon A\to k$ the two canonical projections associated to the direct-sum decomposition $A=S\oplus k$, and writing the elements of $S\oplus k$ as pairs, we then have that $(s,\lambda)(s',\lambda')=(\lambda s'+\lambda' s+\pi_S(ss'),\lambda\lambda'+\pi_k(ss'))$. In this way we arrive very naturally to the notion of $k$-algebra $(S,b_S)$ with a bilinear form $b_S$, that is, a $k$-bilinear mapping $b_S\colon S\times S\to k$, and its associated extension  $(S,b_S)@_{b_S}k$ (Section~\ref{vcou}). Important examples are those of complex numbers, quaternions, octonions, spin factors, Galois extensions of fields, etc.~(Examples~\ref{9} and~\ref{Galois}.)

More precisely, the category of $k$-algebras with a bilinear form is equivalent to the category of unital $k$-algebras $A$ for which the canonical homomorphism $(k,1)\to(A,1_A)$ of unital $k$-algebras is a splitting monomorphism in $k \lMod$ (Theorem~\ref{2}). We call the left inverse of this splitting monomorphism a {\em weak augmentation} of the $k$-algebra $A$. In Section~\ref{3} we consider the compatibility between bilinear forms and multiplication and show that there is a category isomorphism between the category of $k$-algebras with a weak augmentation and the category of unital $k$-algebras with a compatible bilinear form (Theorem~\ref{3'}). In the final section we generalize the ideas of the previous sections to study ring extensions $R \subseteq S$. This kind of extensions also appears in \cite{PJ} and \cite{P}.

	\section{$k$-algebras. Preliminary notions and basic facts.}\label{jbop}
	
	Let $k$ be a commutative ring with identity $1$. With the term {\em $k$-algebra} we mean a $k$-module $R$ with a further $k$-bilinear operation $$\cdot\colon R\times R\to R,\qquad (r,r')\mapsto r\cdot r'=rr'.$$ (Equivalently, with a $k$-module morphism $\mu\colon R\otimes_kR\to R$. See \cite{Bou} and \cite[Section~2]{meltem}). Clearly, $k$-algebras form a category $k \lAlg$, whose morphisms $f\colon R\to S$ are the $k$-module morphisms $f\colon R\to S$ such that $f(rr')=f(r)f(r')$ for every $r,r'\in R$. {\em Abelian} $k$-algebras are those for which $rr'=0$ for every $r,r'\in R$. The full subcategory of $k \lAlg$ whose objects are all abelian $k$-algebras is clearly isomorphic to the category $k \lMod$ of all $k$-modules, hence is a Grothendieck category. Zero objects of $k \lAlg$ are the $k$-algebras with one element, so that $k \lAlg$ is a pointed category. The kernel of a morphism $f\colon R\to S$ is the inclusion of the subalgebra $\ker(f):=\{\,r\in R\mid f(r)=0_S\,\}$ of $R$ into $R$ itself. It is easy to see that $\ker(f)$ is in fact an ideal of $R$. Here a {\em subalgebra} (an {\em ideal}) of $R$ is a $k$-submodule $X$ of $R$ such that $xy\in X$ for every $x,y\in X$ (such that $rx\in X$ and $xr\in X$ for every $r\in R$ and $x\in X$, respectively).
	See~\cite{tim van der linde}.
	
	The $k$-algebra $k$ is a rather special object in the category $k \lAlg$, for instance it is commutative, associative, and has the identity $1$. An {\em identity} in a $k$-algebra $R$ is an element, which we will denote by $1_R$, such that $r\cdot1_R=1_R\cdot r=r$ for every $r\in R$. If $R$ has an identity, we will say that $R$ is {\em unital}. An element $r$ of $R$ is {\em idempotent} if $r^2:=r\cdot r=r$ (since $R$ is not-necessarily associative, $r^n$ is not defined for $n\ge 3$, but it is defined for $n=1$ and $n=2$). The zero of $R$ is always an idempotent element of $R$ (because of the $k$-bilinearity of multiplication), and the identity, when it exists, is also an idempotent element of $R$.
	
	For any $k$-algebra $R$, there is a one-to-one correspondence between the set $\Hom_{k \lAlg}(k,R)$ and the set of all idempotent elements of $R$. For any idempotent element $r$ of $R$ the corresponding morphism $\varphi_r\colon k\to R$ is defined by $\varphi_r(\lambda)=\lambda r$ for every $\lambda \in k$. Conversely, for any morphism $\varphi\colon k\to R$ the corresponding idempotent element of $R$ is $\varphi(1)$.
	
	In particular, when the $k$-algebra $R$ has an identity $1_R$, for $r=1_R$ we have a morphism $\varphi_{1_R}\colon k\to R$ is defined by $\varphi_{1_R}(\lambda)=\lambda {1_R}$ for every $\lambda \in k$. The image of this morphism $\varphi_{1_R}$ is necessarily contained in the center of $R$.
	
	Now let  $k \lAlg_1$ be the category of all unital $k$-algebras. Its morphisms $f\colon R\to S$ are the $k$-algebra morphisms $f$ such that $f(1_R)=1_S$. There is also a third category involved. It is the category $k \lAlg_a$ of all unital $k$-algebras with an {\em augmentation}. Its objects are all the pairs $(A,\varepsilon_A)$, where $A$ is a unital $k$-algebra and $\varepsilon_A\colon A\to k$ is a morphism in $k \lAlg_1$ that is a left inverse of $\varphi_{1_A}$. The morphisms $f\colon (A,\varepsilon_A)\to (B,\varepsilon_{B})$ are the morphisms $f\colon A\to B$ in $k \lAlg_1$ such that $\varepsilon_{B}f=\varepsilon_A$. 
	
	\bigskip
	
	It is well known that:
	
	\begin{theorem}\label{1} There is a category equivalence $F\colon k \lAlg\to k \lAlg_a$ that associates with any object $R$ of $k \lAlg$ the $k$-algebra with augmentation $$F(R):=(R\# k, \pi_2),$$ where $R\#k$  is the $k$-module direct sum $R\oplus k$ with multiplication defined by
$$(r,\lambda)(r',\lambda')=(rr'+\lambda r'+\lambda' r, \lambda\lambda')$$ for every $(r,\lambda)(r',\lambda')\in  R\# k$, and the augmentation $\pi_2\colon R\# k=R\oplus k\to k$ is the second canonical projection. \end{theorem}

The $k$-algebra $R\# k$ is clearly a unital $k$-algebra with augmentation. The quasi-inverse of the equivalence $F$ is the functor $k \lAlg_a\to k \lAlg$ that associates with each $k$-algebra with augmentation $(A,\varepsilon_A)$ the kernel $\ker(\varepsilon_A)$ of the augmentation, which is an ideal, hence a $k$-subalgebra, of $A$.

\medskip

Notice that, for every $k$-algebra $R$,  there is a short exact sequence $$0\to R\to R\# k\to k\to0$$ in the category $k \lAlg$, which is a pointed category with zero objects the algebras with one element. In particular, $R$ is an ideal of the unital $k$-algebra $R\# k$.

As usual, all the identities seen above can be expressed via diagrams. For instance, a $k$-module morphism $f\colon R\to S$ is a $k$-algebra morphism if and only if the diagram $$\xymatrix{
R\otimes_kR\ar[r]^{\mu_R}\ar[d]_{f\otimes f}&R\ar[d]^f \\
S\otimes_kS\ar[r]_{\mu_S} &S }$$ commutes; if, moreover, $R$ and $S$ are unital, then $f$ is a morphism in $k \lAlg_1$ if and only if the diagram $$\xymatrix{ & R\ar[dd]^{f} \\
k \ar[ur]^{\varphi_{1_R}} \ar[dr]_{\varphi_{1_S}} & \\
& S}$$ commutes; also, $f$ is a morphism in $k \lAlg_a$ if and only if the diagram $$\xymatrix{ & R\ar[dd]^{f} \ar[dr]^{\varepsilon_R} &\\
k \ar[ur]^{\varphi_{1_R}} \ar[dr]_{\varphi_{1_S}} & & k\\
& S \ar[ur]_{\varepsilon_S} &}$$ commutes.

\medskip

Theorem \ref{1} explains another reason why the $k$-algebra $k$ is special: a zero object in the category $k \lAlg$ is any zero algebra with one element, so that the zero object in the category $k\lAlg_a$ is the augmented $k$-algebra $F(0)=(k,\id_k)$, that is, the $k$-algebra $k$ is the image of this augmented algebra via the forgetful functor $k\lAlg_a\to k\lAlg_1$. The functor $(-)\#k\colon k \lAlg\to k\lAlg_1$ is  the left adjoint of the forgetful functor $k\lAlg_1\to k\lAlg$.

\bigskip

For any $k$-algebra with augmentation $(A,\varepsilon_A)$, we have $k$-algebras morphisms $
 \xymatrix@1{ k\ar[r]^{\varphi_{1_A}} &
 A\ar[r]^{\varepsilon_A} &
k}$
 whose composite mapping is the identity of $k$. Hence the composite mapping $
 \xymatrix@1{ 
 A\ar[r]^{\varepsilon_A} &
k\ar[r]^{\varphi_{1_A}} & A}$ is an idempotent endomorphism of the $k$-algebra~$A$.

\begin{theorem} For any $k$-algebra $A$ with identity $1_A$, there is a one-to-one correspondence between the set of all idempotent $k$-algebra endomorphisms of $A$ with image $k\cdot 1_A$ and the set of all augmentations of $A$.
\end{theorem} 

For example, let $k$ be a field and $A$ be the ring $k[x]$ of polynomials in one indeterminate $x$ and coefficients in $k$. This is a unital $k$-algebra. For every element $\alpha\in k$, there is an idempotent $k$-algebra endomorphism $v_\alpha\colon k[x]\to k[x]$, $v_\alpha\colon f(x)\mapsto f(\alpha)$, of the $k$-algebra $k[x]$. All these idempotent endomorphisms have image $k$. The augmentation corresponding to the idempotent endomorphism $v_\alpha$ is the augmentation $k[x]\to k$, $f(x)\in k[x]\mapsto f(\alpha)\in k$.

\section{$k$-algebras with a bilinear form}\label{vcou}

	The construction of the unital $k$-algebra $R\# k$ with augmentation is very convenient, but there is a further useful generalization of this construction. Assume that the $k$-algebra $R$ is also endowed with a bilinear form ($=\,k$-bilinear mapping) $b_R\colon R\times R\to k.$ Let  $(R,b_R)@_{b_R} k$ be the $k$-algebra that, as a $k$-module, is the direct sum $R\oplus k$, and with multiplication  defined by \begin{equation}(r,\lambda)(r',\lambda')=(rr'+\lambda r'+\lambda'r, \lambda\lambda'-b_R(r,r').\label{E}\end{equation} The construction of $(R,b_R)@_{b_R} k$ can be viewed as a sort of deformation of the previously constructed $k$-algebra $R\# k$, perturbated by the $k$-bilinear mapping $b_R$. Clearly, when $b_R$ is the zero bilinear mapping, then $(R,b_R)@_{b_R}k$ is exactly $R\# k$.
	
	Let us show, giving some examples, that this construction $-@k$  is very frequent in Algebra.
	
	\begin{examples}\label{9}{\rm 
  (a)  {\sl Complex numbers.} Let $k$ be the field $\R$ of real numbers. As far as $R$ is concerned, let $R$ be the one-dimensional real abelian algebra, so that $R$ can be seen as $\R$ itself but with the algebra multiplication always equal to zero. As a bilinear form $b_R\colon R\times R\to k$ take the usual multiplication in $\R$. Then $(R,b_R)@_{b_R}k$ is exactly the $\R$-algebra of complex numbers.  
  
 (b) {\sl Quaternions.} As in (a), let $k$ be the field $\R$ of real numbers. Now let $R$ be the three-dimensional real $\R$-vector space of Physics. As $\R$-algebra multiplication on $R$ take the cross product $\times$. As bilinear form $b_R\colon R\times R\to k$  take the usual scalar product of vectors. Then $(R,b_R)@_{b_R}k$ now turns out to be the division $\R$-algebra of quaternions. 
 
    (c) {\sl Octonions.} The construction of octonions is also very similar. Let us recall the standard notation concerning octonions. The octonions form a vector spaces $\OO$ over $k:=\R$ of dimension $8$, with basis formed by the {\em unit octonions} $e_0,e_1,e_2,\dots,e_7$. Hence every octonion $x$ can be written in a unique way as a linear combination $x=\sum_{i=0}^7x_ie_i$ with real coefficients $x_i$. The {\em conjugate} of the octonion $x$ is $x^*:=x_0e_0-\sum_{i=1}^7x_ie_i$. The {\em real part} of $x$ is $x_0e_0$, and the {\em imaginary part} of $x$ is $\sum_{i=1}^7x_ie_i$. The set of all purely imaginary octonions span a $7$-dimensional subspace of $\OO$, denoted by $\Imm(\OO)$. Define a ``new multiplication'' on $\Imm(\OO)$ via the multiplication table
    
    \bigskip

  \begin{center}
    \begin{tabular}{|c|c|c|c|c|c|c|c|}
    \hline
         $e_ie_j$ & $e_1$ & $e_2$ & $e_3$ & $e_4$ & $e_5$ & $e_6$ & $e_7$  \\
     \hline
         $e_1$ & $0$ & $e_3$ & $-e_2$ & $e_5$ & $-e_4$ & $e_7$ & $e_6$  \\
    \hline
         $e_2$ & $-e_3$ & $0$ & $e_1$ & $e_6$ & $e_7$ & $-e_4$ & $-e_5$  \\
    \hline
         $e_3$ & $e_2$ & $-e_1$ & $0$ & $e_7$ & $-e_6$ & $e_5$ & $-e_4$  \\
    \hline
         $e_4$ & $-e_5$ & $-e_6$ & $-e_7$ & $0$ & $e_1$ & $e_2$ & $e_3$  \\
    \hline
         $e_5$ & $e_4$ & $-e_7$ & $e_6$ & $-e_1$ & $0$ & $-e_3$ & $e_2$  \\
    \hline
         $e_6$ & $e_7$ & $e_4$ & $-e_5$ & $-e_2$ & $e_3$ & $0$ & $-e_1$  \\
    \hline
         $e_7$ & $-e_6$ & $e_5$ & $e_4$ & $-e_3$ & $-e_2$ & $e_1$ & $0$  \\
       \hline      
    \end{tabular}
  \end{center}
   
    \bigskip
    
    \noindent and extending by linearity. Notice that this multiplication table is skew-symmetric, so that the multiplication on $\Imm(\OO)$ turns out to be anticommutative. More precisely, this is the multiplication table of the {\em seven-dimensional cross product}, a bilinear operation on vectors in the seven-dimensional Euclidean space with the properties of orthogonality and magnitude \cite{BG, MM}. As bilinear form $b\colon \Imm(\OO)\times \Imm(\OO)\to \R$ take the usual Euclidean scalar product. Now $(\Imm(\OO),b)@_{b}\R$ now turns out to be the $\R$-algebra of octonions. 

These first three examples should not surprise. It is a famous result by Hurwitz (\cite{H}, or see \cite[7.6]{J}) that the only composition $k$-algebras when $k$ is a field of characteristic $\ne 2$ are $k$, quadratic algebras, quaternions, and octonions. See \cite[p.~447]{J} and \cite[Theorem I and its proof]{MM}. In these cases, the field $k$ is always the field of real numbers, the bilinear form is always the dot product in an Euclidean space, and the multiplication in the $\R$-algebra is always the cross product. Our construction of $(R,b_{R})@_{b_R}k$ in this special case is called by Beno Eckmann ``a well-known elementary construction'' \cite[p.~338, at the top of the page]{BE}. Notice that in these classical examples, the norm is always multiplicative, that is, the mapping $R\to k$, defined by $r\mapsto b_R(r,r)$ for all $r\in R$, is a magma morphism (i.e., it preserves multiplication).

   (d) Composition algebras are examples of algebras with a bilinear form.

(e) Jordan, von Neumann and Wigner \cite{JNW} proved that every finite-dimensional formally real Jordan algebra is a direct sum of simple ones. There are five families of simple finite-dimensional formally real Jordan algebras: the $n\times n$ self-adjoint real (complex, quaternionic, octonionic for $n\le 3$, respectively) matrices with product $a\circ b=\frac{1}{2}(ab+ba)$, plus the so called {\em spin factors}. These are the real vector spaces $\R^n\oplus \R$ with the product $(x,t)\circ(x',t')=(tx'+t'x, x\cdot x'+tt')$. They are clearly examples of our construction applied to the real $n$-space $\R^n$, viewed as an abelian $\R$-algebra, endowed of the opposite of the dot product. These Jordan algebras play a role in coulor perception, see \cite{BP}.

(f) Let $k$ be a field, and $R$ be an abelian $k$-algebra with a bilinear form $b_R$. Let $Q$ be the quadratic form associated with $b_R$, so that $Q(r)=b_R(r,r)$ for every $r\in R$. Then $(R,{b_R})@_{b_R}k$ is an associative unital algebra, and the inclusion $i\colon R\to (R,{b_R})@_{b_R}k$ is such that $i(r)^2=(r,0)(r,0)=(0,-b(r,r))=-Q(r)1_{(R,{b_R})@_{b_R}k}$. By the universal property of Clifford algebras, there is a unique $k$-algebra morphism of the Clifford algebra Cl$(R,-Q)$ onto $(R,{b_R})@_{b_R}k$. This will be generalized and improved in Theorem~\ref{associative}(c) and Proposition~\ref{na}.

(g) {\em Quadratic algebras} are obtained doubling a field $k$ \cite[page 446]{J}. Hence, suppose that $k$ is a field. Let $R$ be the one-dimensional abelian $k$-algebra, so that $R$ can be vewed as $k$ with zero multiplication. Fix $c\in k$ and let $k\colon R\times R\to k$ be the $k$-bilinear form defined by $b(x,y)=-cxy$. Then the product in $R@_kk$ is defined by $(x,\lambda)(y,\mu)=(\lambda y+\mu x, \lambda\mu+cxy)$, i.e., $R@_kk$ is exactly the $c$-double of $k$, that is, the quadratic algebra obtained doubling the unital $k$-algebra $k$, with the identity as involution, and $c\in k$ \cite[page 444, (94)]{J}. }\end{examples}

	\bigskip
	
	Let us try to be more precise. Fix a commutative unital ring $k$. Let $ k\BFlAlg$ be the category of all  {\em $k$-algebras with a bilinear form}, whose objects are all the pairs $(R,b)$, where $R$ is a $k$-algebra and $b\colon R\times R\to k$ is a $k$-bilinear mapping (equivalently, $b\colon R\otimes_k R\to k$ is a $k$-linear mapping). The morphisms $\varphi\colon (R,b_R)\to (S,b_S)$ in this category are the $k$-algebra morphisms $\varphi\colon R\to S$ that are {\em orthogonal transformations}, that is, $b_R(r,r')=b_s(\varphi(r),\varphi(r'))$ for every $r,r'\in R$. As usual, this identity can be expressed in term of commutativity of the diagram

$$\xymatrix{  R\otimes_kR\ar[dd]_{\varphi\otimes\varphi}  \ar[dr]^{b_R} & \\
& k. \\
S\otimes_kS\ar[ur]_{b_S}} $$
	
	\bigskip
	
	Similarly, there is a category $ k\BFlAlg_1$ of all  {\em unital $k$-algebras with a bilinear form}, whose objects are all pairs $(R,b)$, where $R$ is now also required to be a $k$-algebra with identity $1_R$. The morphisms $\varphi\colon (R,b_R)\to (S,b_S)$ in the category $ k\BFlAlg_1$ are required to send $1_R$ to $1_S$. Finally, there is a category $k \lAlg_{wa}$ of all $k$-algebras $A$ with an identity $1_A$ and with a {\em weak augmentation}, that is, a morphism $\varepsilon_A\colon A\to k$ in $k \lMod$ that is a left inverse of $\varphi_{1_A}$. Notice that the difference between augmentation and weak augmentation is that augmentations are morphisms of $k$-algebras, while weak augmentations are only morphisms of $k$-modules. We will write the objects of $k \lAlg_{wa}$ as pairs $(A, \varepsilon_A)$. The morphisms $f\colon (A,\varepsilon_A)\to (B,\varepsilon_{B})$ are the morphisms $f\colon A\to B$ in $k \lAlg_1$ such that $\varepsilon_{B}f=\varepsilon_A$. 
		
\begin{remark}{\rm In the definition of $k$-algebra with a bilinear form given above, we don't require any form of compatibility between the multiplication of the $k$-algebra and the bilinear form. Compatibility axioms will be studied in Section~\ref{3}.}\end{remark}

\begin{lemma}\label{2.2} Let $(R,b_R)$ and $ (S,b_S)$ be two $k$-algebras with a bilinear form, and $\psi\colon R\to S$ be a $k$-module morphism. Then $\varphi=:(\psi\oplus\id_k)\colon (R,b_R)@_{b_R}k\to (S,b_S)@_{b_S}k$ is a morphism of $k$-algebras if and only if $\psi\colon (R,b_R)\to (S,b_S)$ is a morphism of $k$-algebras with a bilinear form. Moreover, if these equivalent conditions are satisfied, then $\varphi$ is necessarily a morphism of unital $k$-algebras.\end{lemma}

\begin{proof} The $k$-module morphism $\varphi$ is a $k$-algebra morphism if and only if  $\varphi$ preserves $k$-algebra multiplication, that is, if and only if $$(\psi\oplus\id_k)((r,\lambda)(r',\lambda'))=(\psi\oplus\id_k)(r,\lambda)(\psi\oplus\id_k)(r',\lambda'),$$ that is, if and only if $$\begin{array}{l}(\psi\oplus\id_k)(rr'+\lambda r'+\lambda'r, \lambda\lambda'-b_R(r,r'))= \\ \qquad =(\psi(r)\psi(r')+\lambda \psi(r')+\lambda'\psi(r), \lambda\lambda'-b_S(\psi(r),\psi(r'))).\end{array} $$ It is now easy to see that this occurs if and only if $\psi$ is a $k$-algebra morphism and $b_R=b_S\circ(\psi\otimes\psi)$.\end{proof}

	\begin{theorem}\label{2} There is an equivalence of categories $G\colon k \BFlAlg\to k \lAlg_{wa}$ that associates with any object $(R,b_R)$ of $k \BFlAlg$ the $k$-algebra with a weak augmentation $G(R,b_R):=((R,{b_R})@_{b_R} k,\pi_2)$, where $(R,{b_R})@_{b_R}k$ is the $k$-module direct sum $R\oplus k$ with multiplication defined by
$$(r,\lambda)(r',\lambda')=(rr'+\lambda r'+\lambda'r, \lambda\lambda'-b_R(r,r'))$$ for every $(r,\lambda),(r',\lambda')\in  (R,{b_R})@_{b_R} k$, with identity $(0_R,1_k)$, and with weak augmentation the second canonical projection $$\pi_2\colon (R,{b_R})@_{b_R}k=R\oplus k\to k.$$ \end{theorem}

\begin{proof} Let $(R,b_R)$ be an object of $ k\BFlAlg$. The $k$-algebra $(R,b_R)@_{b_R}k$ is clearly a unital $k$-algebra with weak augmentation $$\pi_2\colon (R,b_{R})@_{b_R}k=R\oplus k\to k.$$ The functor $G$ associates with any object $(R,b_R)$ of $k \BFlAlg$ the object\linebreak $((R,b_{R})@_{b_R} k,\pi_2)$ of $k \lAlg_{wa}$, and with any morphism $\varphi\colon (R,b_R)\to (S,b_S)$ in $k \BFlAlg$ the mapping $\varphi\oplus\id_k\colon (R,b_{R})@_{b_R}k\to( S,b_S)@_{b_S}k$. This mapping is clearly a $k$-module morphism, but is also a morphism of $k$-algebras with identities by Lemma~\ref{2.2}. 

The quasi-inverse of  $G$ is the functor $H\colon k \lAlg_{wa}\to k \BFlAlg$ that associates with each $k$-algebra with a weak augmentation $(A,\varepsilon_A)$ the kernel $K:=\ker(\varepsilon_A)$. More precisely, from $\varphi_{1_A}\circ\varepsilon_A=\id_k$,  we get that $A=\ker(\varepsilon_A)\oplus\Imm(\varphi_{1_A})=K\oplus k\cdot 1_A$, so that there are the two canonical projections $\pi_1\colon A=K\oplus k\cdot 1_A\to K$ and $\pi_2\colon A=K\oplus k\cdot 1_A\to k\cdot 1_A$. Let $\mu\colon A\times A\to A$ be the multiplication in $A$ and $\mu|_{K\times K}\colon K\times K\to A$ be the restriction of $\mu$ to $K\times K$. The $k$-module $K$ is a $k$-algebra with respect to the multiplication \begin{equation}\pi_1\circ \mu|_{K\times K}\colon K\times K\to K.\label{*}\end{equation} Since $\varphi_{1_A}\colon k\to A$ is an injective mapping with image $k\cdot 1_A$, the corestriction $\varphi_{1_A}|^{k\cdot 1_A}\colon k\to k\cdot 1_A$, $\varphi_{1_A}|^{k\cdot 1_A}(\lambda)=\lambda 1_A$, is a $k$-algebra isomorphism. Set $b_K:=\left(\varphi_{1_A}|^{k\cdot 1_A}\right)^{-1}\circ\pi_2\circ\mu|_{K\times K}$. Then $H(A,\varepsilon_A):=(K,b_K)$ is a $k$-algebra with a bilinear form. 

As far as morphisms are concerned, let $\varphi\colon (A,\varepsilon_A)\to (B,\varepsilon_B)$ be a morphism of $k$-algebras with a weak augmentation. Then $\varepsilon_B\circ\varphi=\varepsilon_A$, so that $\varphi$ maps $$K_A:=\ker(\varepsilon_A)\qquad\mbox{\rm into}\qquad K_B:=\ker(\varepsilon_B)$$ and maps $1_A$ to $1_B$. Hence $\varphi$ induces by restriction a $k$-module morphism $$\varphi|_{K_A}^{K_B}\colon K_A\to K_B$$ and a $k$-algebra isomorphism $$k\cdot 1_A\to k\cdot 1_B, \qquad\lambda 1_A\mapsto\lambda 1_B,$$ which is the $k$-algebra isomorphism $(\varphi_{1_B}|^{k\cdot 1_B})\circ (\varphi_{1_A}|^{k\cdot 1_B})^{-1}$.
It follows that $\varphi\colon A=K_A\oplus k\cdot 1_A\to B=K_B\oplus k\cdot 1_B$ can be written as $$\varphi=\left(\varphi|_{K_A}^{K_B}\oplus\left((\varphi_{1_B}|^{k\cdot 1_B})\circ (\varphi_{1_A}|^{k\cdot 1_B})^{-1}\right)\right).$$ From Lemma~\ref{2.2}, we know that $\varphi|_{K_A}^{K_B}$ is a morphism of $k$-algebras with a bilinear form. The functor $H$ associates with any morphism $\varphi\colon (A,\varepsilon_A)\to (B,\varepsilon_B)$ of algebras with an augmentation the morphism $\varphi|_{K_A}^{K_B}$ of $k$-algebras with a bilinear form. Of course, the multiplications on ${K_A}$ and ${K_B}$ are $\pi_1\circ \mu|_{K_A\times K_A}\colon K_A\times K_A\to A=K_A\oplus k\cdot 1_A\to K_A$ and $\pi_1\circ \mu|_{K_B\times K_B}\colon K_B\times K_B\to B=K_B\oplus k\cdot 1_B\to K_B$ respectively, as in (\ref{*}). It is now easy to check that $G$ and $H$ define a category equivalence.
\end{proof}

Notice that there is a short exact sequence $0\to R\to R@ k\to k\to0$ in the category $k \lMod$. Hence $R$ is a $k$-submodule of $R@ k$.

\begin{remark}{\rm We have already remarked that the functor $(-)\#k\colon k \lAlg\to k\lAlg_1$ is  the left adjoint of the forgetful functor $k\lAlg_1\to k\lAlg$. A natural question is whether this remains true in the more general setting of the functor $(-)@k$. Now if the functor $(-)@k$ is the left adjoint of some forgetful functor between suitable categories of $k$-algebras, then the adjunction would have a unit $\eta$, that is, for every $k$-algebra $R$ with a bilinear form $b$ there would be a canonical $k$-algebra morphism $\eta_{(R,b)}\colon R\to R@_bk$. Now it is easy to check that in general there are no $k$-algebra morphisms $R\to R@_bk$, except for some very special cases, for instance:

(a) the case of $b=0$, so that $R@_bk=R\#k$, and the $k$-algebra morphism $R\to R@_bk=R\#k=R\oplus k$ is the canonical mapping $r\mapsto(r,0)$;

(b) the case in which $R$ is an augmented $k$-algebra. If $\varepsilon_R\colon R\to k$ is an augmentation (= a morphism of $k$-algebras), then there is a $k$-algebra morphism $R\to R@_bk$, $r\mapsto (0,\varepsilon_R(r))$. 

Notice that $R@_bk$ does not contains subalgebras isomorphic to $R$ in general, but it does contain a subalgebra isomorphic to $k$. 

Since there are no $k$-algebra morphisms $R\to R@_bk$ in the general case, the functor $(-)@k$ is not the left adjoint of a forgetful functor between suitable categories of $k$-algebras.}\end{remark}

Also notice that in the category $k\BFlAlg$ there is also an analogue of the conjugation of complex numbers, quaternions, etc. It is the mapping $c\colon (R,\mu,b)\to (R,-\mu,b)$, that associates to any $k$-algebra $R$ with multiplication $\mu$ and bilinear form $b$ the $k$-algebra $R$ with multiplication $-\mu$ and bilinear form $b$. Here $-\mu$ is the multiplication in which the product of $r,s\in R$ is $-rs$. 

We will see in the following that our results can be also applied to ring extensions (Section~\ref{4000}) or additive categories in which idempotent endomorphisms have kernels (Remark~\ref{4.1}).

\section{Compatibility between bilinear forms and multiplication}\label{3}

Let $A$ be a $k$-algebra. We say that a $k$-bilinear form $b\colon A\times A\to k$ is {\em compatible with the multiplication of $A$} if, for every $x,y,z,w\in A$, $xy=zw$ implies $b(x,y)=b(z,w)$. 

\begin{lemma} Let $A$ be a $k$-algebra with multiplication $\mu\colon A\times A\to A$ and $b\colon A\times A\to k$ a $k$-bilinear form. Consider the following conditions:

\noindent{\rm (a)}  $b$ is compatible with the multiplication of $A$.

\noindent{\rm (b)} $b$ factors through $\mu$, that is, there exists a $k$-module morphism $\varepsilon\colon A\to K$ such that $\varepsilon\circ\mu=b$.

Then {\rm (b)} ${}\Rightarrow{}${\rm (a)}. 

If, moreover, $A$ is unital, then {\rm (a)} ${}\Leftrightarrow{}${\rm (b)}, and if these equivalent conditions {\rm (a)} and {\rm (b)} are satisfied then the $k$-module morphism $\varepsilon$ in {\rm (b)} is unique.\end{lemma}

\begin{proof} If $b$ factors through $\mu$, that is, $\varepsilon\circ\mu=b$ for some $k$-module morphism $\varepsilon\colon A\to K$,  $x,y,z,w\in A$ and $xy=zw$, then $$b(x,y)=(\varepsilon\circ\mu)(x,y)=\varepsilon(xy)=\varepsilon(zw)=(\varepsilon\circ\mu)(z,w)=b(z,w).$$ This shows that {\rm (b)} ${}\Rightarrow{}${\rm (a)}. 

Now assume that $A$ is unital. If $b$ is compatible with the multiplication of $A$, define $\varepsilon\colon A\to K$ setting $\varepsilon(a)=b(a,1_A)$ for every $a\in A$. Then $\varepsilon\circ\mu=b$ because, for every $a,a'\in A$, $(\varepsilon\circ\mu)(a,a')=\varepsilon(aa')=b(aa',1_A)=b(a,a')$ because $b$ is compatible with the multiplication. Hence {\rm (a)} ${}\Leftrightarrow{}${\rm (b)}.
Moreover, the mapping $\mu$ is surjective, so that $\varepsilon\circ\mu=b=\varepsilon'\circ\mu$ implies $\varepsilon=\varepsilon'$.
\end{proof}

\begin{lemma}\label{4} Let $A$ be a $k$-algebra with a $k$-bilinear form $b\colon A\times A\to k$. Consider the following two conditions:

{\rm (a)}  $b$ is compatible with the multiplication of $A$.

{\rm (b)} $b(xy,z)=b(x, yz)$ for every $x,y,z\in A$.

\noindent Then:

{\rm (1)} If $A$ is associative, then {\rm (a)}${}\Rightarrow{}${\rm (b)}.

{\rm (2)} If $A$ is unital, then {\rm (b)}${}\Rightarrow{}${\rm (a)}.\end{lemma}

\begin{proof} (1) If $A$ is associative and  $x,y,z\in A$, then $(xy)z=x(yz)$, so that (b) holds by (a).

{\rm (2)} If $A$ has an identity $1_A$, say, $x,y,z,w\in A$ and  $xy=zw$, then $b(x,y)=b(x,y\cdot 1_A)=b(xy,1_A)=b(zw,1_A)=b(z,w\cdot 1_A)=b(z,w)$.\end{proof}

In the following, for a commutative unital  ring $k$ and a $k$-module $A$, we will write 
$A^{\otimes n} $ for the $n$-fold tensor product of $A$ over $k$ and $$T(A):=\bigoplus_{n\ge  0} A^{\otimes n} =k\oplus A\oplus (A\otimes A)\oplus (A\otimes A\otimes A)\oplus\dots$$ to denote the tensor algebra.

\begin{theorem}\label{associative} The following conditions are equivalent for a $k$-algebra $(R,b)$ with multiplication $\mu_R\colon R\times R\to R$ and a bilinear form $b$:

{\rm (a)}  The $k$-algebra $(R,b)@_bk$ is associative.

{\rm (b)}  $b(xy,z)=b(x, yz)$ and
$(xy)z-b(x,y)z=x(yz)-xb(y,z)$  for every $x,y,z\in R$.

{\rm (c)}  The $k$-algebra $(R,b)@_bk$ is isomorphic to the tensor algebra $T(R)$ modulo the two-sided ideal $I$ of $T(R)$ generated by the set of all the elements $b(x,y)-\mu_R(x,y)+x\otimes y$, where $x,y\in R$.\end{theorem}

\begin{proof} {\rm (a)}${}\Leftrightarrow{}${\rm (b)} Computing $((x,\lambda)(y,\mu))(z,\nu)$ and $(x,\lambda)((y,\mu)(z,\nu))$ one sees that they are equal if and only if the two equalities in (b) hold.

{\rm (c)}${}\Rightarrow{}${\rm (a)} The tensor algebra $T(R)$ is associative for every $k$-module $R$, so that if $(R,b)@_bk$ is a homomorphic image of $T(R)$, then $(R,b)@_bk$ is associative.

{\rm (a)}${}\Rightarrow{}${\rm (c)} Tensor algebra has the universal property with respect to associative unital algebras, so that if the $k$-algebra $(R,b)@_bk$ is associative, the $k$-module morphism $\varepsilon_1\colon R\to (R,b)@_bk=R\oplus k$ can be uniquely extended to a unital $k$-algebra morphism $T(R)\to (R,b)@_bk$. This $k$-algebra morphism maps the elements of $T(R)$ of the form $\lambda+x$ of $T(R)$ to the elements $(\lambda,x)$ of 
$(R,b)@_bk$, and the elements of $T(R)$ of the form $b(x,y)-\mu_R(x,y)+x\otimes y$ to $$\begin{array}{l}b(x,y)-\mu_R(x,y)+\mu_{R@k}(x, y)=(0,b(x,y))+(-xy,0)+(x,0)(y,0)= \\ \qquad=(0,b(x,y))+(-xy,0)+(xy,-b(x,y))=0.\end{array}$$  Hence there is an induced surjective  $k$-algebra morphism $T(R)/I\to (R,b)@_bk$. Clearly, every element of the graded $k$-algebra $T(R)$ is congruent modulo $I$ to an element of degree $\le 1$, that is, to an element of $T(R)$ of the form $\lambda+x$ (this can be seen easily by induction on the degree of the element of $T(R)$). Such an element is mapped to the zero of $(R,b)@_bk$ if and only if $\lambda=0$ and $x=0$. Hence the morphism $T(R)/I\to (R,b)@_bk$ is an isomorphism.
\end{proof}

\begin{remarks} {\rm (a) It is easy to check that $(R,b)@_bk$ is commutative if and only if $R$ is commutative and $b$ is symmetric, that is, $b(x,y)=b(y,x)$ for all $x,y\in R$.

(b) If the $k$-algebra $(R,b)@_bk$ is associative, it is not necessarily true that $R$ is associative. Cf.~Theorem~\ref{associative}. For instance, in Example~\ref{9} (b) (quaternions), the three-dimensional real $\R$-vector space $R$ of Physics with the cross product $\times$ is not associative, but the algebra $(R,b_R)@_{b_R}\R$ of quaternions is an associative division algebra.}\end{remarks}

The representation $T(R)/I\cong( R,b)@_bk$ of Theorem~\ref{associative}(c), can be extended to non-associative $k$-algebras. To this end, we must replace the tensor algebra $T(A)$, which is associative, with its non-associative version. In order to present it, we need a digression.

Recall that the free non-associative algebra on a set $X$ over a commutative unital ring $k$ is the free $k$-module with free set of generators consisting of all non-associative monomials (formal products of elements of $X$ retaining parentheses). The product of monomials $u,v$ is just $(u)(v)$. This algebra is unital (take the empty product as a monomial). 

We are mainly interested in the case in which $X$ has only one element $x$. In this case, the free non-associative $k$-algebra on the indeterminate $x$ is the free $k$-module with free set of generators of all non-associative monomials in $x$. Similarly, we have the {\em free cyclic unital magma} $S$: it is the set of all non-associative monomials in an indeterminate $x$, so that \begin{multline}S:=\{1, x, x\cdot x, x\cdot (x\cdot x), (x\cdot x)\cdot x, 
(x\cdot x)\cdot( x\cdot x), x\cdot (x\cdot (x\cdot x)), \\ x\cdot ((x\cdot x)\cdot x), (x\cdot (x\cdot x))\cdot x, (( x\cdot x)\cdot x)\cdot x, \dots\}.\label{xxx}\end{multline} The free non-associative $k$-algebra on the indeterminate $x$ is the free $k$-module with $S$ as a free set of generators. Any element $s$ of $S$ has a {\em degree} $d(s)$, which is a non-negative integer. In describing the (``first'') elements of $S$ in (\ref{xxx}) we have written the unique element of $S$ of degree $0$, the unique element of $S$ of degree $1$, the unique element of $S$ of degree $2$, the $2$ elements of $S$ of degree $3$, and the $5$ elements of $S$ of degree $4$. 

The {\em non-associative tensor algebra} of a $k$-module $R$, where $k$ is a commutative unital ring, is $$T_{na}(R):=\bigoplus_{s\in S}R^{\otimes d(s)} .$$ Notice how this extends the associative case. In the associative case we have that $$T(R):=\bigoplus_{n\in \NN}R^{\otimes n} .$$ Here $\NN$ is the free cyclic monoid, i.e.~the associative analogue of the free cyclic unital magma $S$. In $T_{na}(R)$, we have one copy of $R^{\otimes n} $ ($n$ times) for each element $s\in S$ of degree $n$, i.e., for each way of writing the parentheses in a monomial in $x$ of degree $n$.

The non-associative $k$-algebra $T_{na}(R)$ ($k$ a commutative unital ring, $R$ a $k$-module) has the universal property for our non-associative unital $k$-algebras: Any $k$-module morphism $f\colon R\to A$ from $R$ to 
a unital $k$-algebra $A $ can be uniquely extended to a morphism of unital algebras from $T_{na}(R)$ to $A$. The same proof as Theorem~\ref{associative}((a)${}\Rightarrow{}$(c)) shows that:

\begin{proposition}\label{na} Let $(R,b)$ be a $k$-algebra with a bilinear form $b$. Then $$(R,b)@_bk\cong T_{na}(R)/I,$$ where $I$ is the two-sided ideal of $T_{na}(R)$ generated by the set of all the elements $b(x,y)-\mu_R(x,y)+x\otimes y$ ($x,y\in R$).\end{proposition}

Let $k\BFlAlg_c$ be the full subcategory of $k\BFlAlg_1$ whose objects are all unital $k$-algebras $(A,b_A)$ whose bilinear form $b_A$ is compatible. 

\begin{theorem}\label{3'} Let $k$ be a commutative unital ring. There is a category isomorphism $M\colon k\lAlg_{wa}\to k\BFlAlg_c$ of the category $k\lAlg_{wa}$ of $k$-algebras with a weak augmentation and the category $k\BFlAlg_c$ of unital $k$-algebras with a compatible bilinear form. This functor $M$ associates with any $k$-algebra $(A,\varepsilon_A)$ with a weak augmentation the $k$-algebra $(A,b_A)$ with compatible bilinear form $b_A:=\varepsilon_A\circ \mu_A$. Here $\mu_A\colon A\times A\to A$ is the multiplication of the algebra $A$. The quasi-inverse functor of $M$ associates with any $k$-algebra $(A,b_A)$ with identity $1_A$ and a compatible bilinear form $b_A$ the algebra with weak augmentation $(A,\varepsilon_A)$, where $\varepsilon_A\colon A\to k$ is defined by $\varepsilon_A(a)=b_A(a,1_A)$ for every $a\in A$.\end{theorem}

\begin{proof} We first show that the two assignations are mutually inverse on objects. Let $(A,\mu_A,1_A)$ be a unital $k$-algebra.

If on $A$ there is a weak augmentation $\varepsilon_A$, then the associated bilinear form is $b_A=\varepsilon_A\circ\mu_A$. With this bilinear form, the associated weak augmentation is $\varepsilon'_A(a)=b_A(a,1_A)=\varepsilon_A(a\cdot 1_A)=\varepsilon_A(a)$ for every $a\in A$, so that $\varepsilon_A=\varepsilon_A'$.

Conversely, if $A$ has a compatible bilinear form $b_A$, then the weak augmentation associated with it is $\varepsilon_A(a)=b_A(a,1_A)$, and the bilinear form associated with this weak augmentation is $b'_A=\varepsilon_A\circ\mu_A$, so that $$b'_A(x,y)=\varepsilon_A(xy)=b_A(xy,1_A)=b_A(x,y).$$ This shows that the two assignations define a canonical one-to-one correspondence between weak augmentations on $A$ and compatible bilinear forms on $A$.

As far as $k$-linear mappings are concerned, we must show that, for any two $k$-algebras $(A,\varepsilon_A),(B,\varepsilon_B)$ with weak augmentations, a morphism $\varphi\colon A\to B$ of unital $k$-algebras respects the augmentations if and only if it respects the associated compatible bilinear forms $b_A=\varepsilon_A\circ\mu_A$ and $b_B=\varepsilon_B\circ\mu_B$. Now if $\varphi$ respects the augmentations, that is, if $\varepsilon_B\varphi=\varepsilon_A$, then $\varphi$ respects the associated compatible bilinear forms, because $$b_A=\varepsilon_A\circ\mu_A=\varepsilon_B\circ\varphi\circ\mu_A=\varepsilon_B\circ\mu_B\circ(\varphi\times\varphi)=b_B\circ(\varphi\times\varphi).$$ Conversely, if $\varphi$ respects the associated compatible bilinear forms, i.e., $$b_A=b_B\circ(\varphi\times\varphi),$$ evaluating the equality on the elements $(a, 1_A)$ we get that $$b_A(a, 1_A)=(b_B\circ(\varphi\times\varphi))(a, 1_A)=b_B(\varphi(a),1_B),$$ that is, $\varepsilon_A(a)=\varepsilon_B(\varphi(a))$, as desired.
\end{proof}

Combining the results of Theorems~\ref{2} and \ref{3'}, we see that the functor $$(-)@k\colon k\BFlAlg\to k\BFlAlg_1$$ associates with any $k$-algebra $(R,b_R)$ with a bilinear form the unital $k$-algebra $(R,b_R)@_{b_R}k$, whose bilinear form is $\pi_2\circ\mu_{R@k}$, that is, the orthogonal direct sum $(-b_R)\boxplus 1_k$. This sign minus in front of $b_R$ is a little annoying, but it is necessary in view of our classical Examples~\ref{9}.

\bigskip

\begin{proposition} For any $k$-algebra $A$ with identity $1_A$, there is a one-to-one correspondence between the set of all idempotent $k$-module endomorphisms of $A$ with image $k\cdot 1_A$ and the set of all  weak augmentations of $A$.
\end{proposition} 

\begin{example}\label{Galois}{\rm The most important classical example of weak augmentation is the trace of a field extension. Let $L/K$ be a finite field extension for which the characteristic of $K$ does not divide the degree $[L:K]$ of the extension. Recall that the {\em trace} $\Tr_{L/K}(\alpha)$ of an element $\alpha\in L$ is the trace, in the sense of linear algebra, of multiplication by $\alpha$ viewed as a $K$-vector space endomorphism of the $K$-vector space $L$. When $L/K$ is a Galois extension, the trace of $\alpha$ is the sum of all the elements $\sigma(\alpha)\in L$ where $\sigma$ ranges in the Galois group of $L/K$. If $\lambda\in K$, one has that $\Tr_{L/K}(\lambda)=[L:K]\lambda$. Therefore
the mapping $\varepsilon\colon L\to K$ defined by $\varepsilon(\alpha):=[L:K]^{-1}\Tr_{L/K}(\alpha)$ for every $\alpha\in L$, is a weak augmentation of the $K$-algebra $L$. 

Let $Z$ be the set of all the elements of $L$ of trace $0$, so that $L=Z\oplus K$ as\linebreak $K$-vector spaces. Define the product $\mu_Z$ on $Z$ defining $\mu_Z(\alpha,\beta):=\alpha\beta-\varepsilon(\alpha\beta)$ for every $\alpha,\beta\in Z$, and define the $K$-bilinear form $b_Z$ on $Z$ setting $b_Z(\alpha,\beta)=-\varepsilon(\alpha\beta)$. The $K$-algebra $(Z,\mu_Z,b_Z)$ with bilinear form $b_Z$ is the $K$-algebra with bilinear form that produces the extension $(Z,b_Z)@_{b_Z}K\cong L$, which is a $K$-algebra with weak augmentation $\varepsilon\colon L\to K$. In the notation of Theorem~\ref{3'}, we have that the functor $M$ associates with the $K$-algebra $(Z,\varepsilon|_Z)$ with weak augmentation the restriction $\varepsilon|_Z\colon Z\to K$ of $\varepsilon$ to $Z$, the $K$-algebra $M(Z,\varepsilon|_Z)=L$ with bilinear form $$b':=(-b_Z)\boxplus 1_K\colon L\times L\to K.$$ Notice that multiplication on $L$ and the bilinear form $b'$ are compatible because $b'(\alpha, \beta)=\varepsilon(\alpha\beta)$.

In this example, it is easy to show that the equality $$b_Z(\mu_Z(x,y),z)=b_Z(x,\mu_Z(y,z))$$ of Lemma~\ref{4}(b) holds for every $x,y,z\in Z$. This can be either checked directly, or using Theorem~\ref{associative} and the fact that the multiplication on $L$ is associative. In the next paragraph we will see that the multiplication $\mu_Z$ on $Z$ and the bilinear form $b_Z$ are not compatible in general. Hence the implication (b)$\Rightarrow$(a) of Lemma~\ref{4}(b) does not hold for the $K$-algebra without identity $(Z,\mu_Z)$.
The multiplication $\mu_Z$ on $Z$ is commutative, but not associative in general.

This example of a finite field extension $L/K$ is a generalization of Example~\ref{9}(a) of complex numbers, where the trace $\Tr_{\C/\R}(\alpha)$ of a complex number $\alpha$ is $2$Re$(\alpha)$ (twice the real part of $\alpha$). Hence in our Example~\ref{9}(a) the (zero) multiplication on $Z=i\R$ and the bilinear form $b_Z$ on $Z$ are not compatible.  }\end{example}
\section{Ring extensions}\label{4000}

We now generalize the ideas of the previous sections to study ring extensions $R\subseteq S$. Now $R$ and $S$ are associative ring with identity $1_R=1_S$. Then $S$ is naturally an $R$-$R$-bimodule $_RS_R$, and the inclusion of $R$ into $_RS_R$ is an $R$-$R$-bimodule monomorphism. Let us call {\em augmentation} of the ring extension $R\subseteq S$ any ring morphism $\varepsilon_S\colon S\to {R}$ that is the identity on $R$, and
{\em weak augmentation} of the ring extension $R\subseteq S$ any $R$-$R$-bimodule morphism $\varepsilon_S\colon {}_RS_R\to {}{}_RR_R$ that is the identity on $R$. It is easily checked that every augmentation is a weak augmentation.

There is an overlapping between what we are doing here and \cite{PJ, P}. For instance weak augmentations are considered in \cite[p.~50]{PJ}, under the name of ``right splitting extension''. The only difference with us is that we deal with $R$-$R$-bimodule morphisms, and \cite{PJ} deals with right $R$-module morphisms.

\bigskip

If $R\subseteq S$ are associative rings and $R$ has an identity $1_R$, let $C_S(R)$ be the centralizer of $R$ in $S$. Then $C_S(R)$ is a subring of $S$, and there is a one-to-one correspondence between the set $\Hom_{\Rng} (R,S)$ of all morphisms of $R$ into $S$ in the category $\Rng$ of rings possibly non-unital and the set of all idempotent elements of $C_S(R)$. For any idempotent element $e$ of  $C_S(R)$ the corresponding morphism $\varphi_e\colon R\to S$ is defined by $\varphi_e(r)=e r$ for every $r \in R$. Conversely, for any morphism $\varphi\colon R\to S$ the corresponding idempotent element of $C_S(R)$ is $\varphi(1)$.

It is easy to check that $\varphi_e\colon R\to S$, $\varphi_e\colon r\mapsto e r$, is a right $R$-module morphism for every $e\in S$. For an idempotent $e\in S$, $\varphi_e$ is a ring morphism if and only if it is an $R$-$R$-bimodule morphism, if and only if $e\in C_S(R)$.

If $\varepsilon\colon S\to R$ is an augmentation, then $S=K\oplus R$, where $K:=\ker\varepsilon$ is a two-sided ideal of $S$, and $(k,r)(k',r')=(kr'+rk'+kk', rr')$.

In general, given any unital ring $R$ and any $R$-$R$-bimodule $_RK_R$ that is also a ring (without identity) with respect to a multiplication $\mu_K\colon K\times K\to K$, which is an $R$-balanced mapping (equivalently, there is an $R$-$R$-bimodule morphism $\mu_K\colon {}_RK\otimes_R K_R\to {}_RK_R$), then it is possible to give $K\oplus R$ a suitable ring structure (for the associativity of the multiplication on this new ring $K\oplus R$, some suitable axioms are necessary). This construction is called a {\em trivial extension} in \cite[p.~223]{P} for $K$ abelian, and a {\em semi-trivial extension} when a bilinear form is involved.

More precisely, let $R$ be any unital ring, let $_RK_R$ be an $R$-$R$-bimodule  that is also a ring (without identity) with respect to a multiplication $\mu_K\colon K\times K\to K$, which is an $R$-balanced mapping (equivalently, there is an $R$-$R$-bimodule morphism $\mu_K\colon {}_RK\otimes_R K_R\to {}_RK_R$), and suppose that there is a ``bilinear form $b_K$'', that is, an $R$-$R$-bimodule morphism $b_K\colon {}_RK\otimes_R K_R\to {}_RR_R$. Then it is possible to give $K\oplus R$ a ring structure defining $(k,r)(k',r')=( kr'+rk'+kk',rr'+b(k,k'))$ (again, for the associativity of the multiplication on this  ring $K\oplus R$, some suitable axioms are necessary).  Here we have preferred not to change the sign before $b(k,k')$ to be congruent with \cite{PJ} and \cite{P}.

\begin{remark}\label{4.1}{\rm The categorical setting of our results is the following. Let $C$ be an additive category equipped with a functor
 $T\colon C\times C \to C$,
 additive in each argument, and having an object $K$ with natural and coherent isomorphisms $T(K,A) \to A$ and $T(A,K)\to A$. Then the following two categories are equivalent:
 
{\rm (a)} The category of triples $(A,m,p)$, where $A$ is an object in $C$, and $m\colon T(A,A)\to A$ and $p\colon T(A,A)\to K$ are morphisms.
 
{\rm (b)}  The category of four-tuples $(A,e,m,q)$, where 

\noindent {\rm (1)} $A$ is an object in $C$,  

\noindent {\rm (2)} $m\colon T(A,A)\to A$ is a morphism, 

\noindent {\rm (3)} $e\colon K\to A$ is a morphism making $T(K,A)\to T(A,A)\to A$ equal to the canonical isomorphism $T(K,A)\to A $, and $T(A,K)\to T(A,A)\to A$
 equal to the canonical isomorphism
 $T(A,K)\to A$, 
 and 
 
 \noindent {\rm (4)} $q\colon A\to K$ is a morphism with a kernel and such that $qe=1$.
 
 The corresponding idempotent endomorphisms are endomorphisms with a kernel. Recall that in an additive category, idempotent endomorphisms have a kernel if and only if they split \cite[Proposition~4.17]{libro}.}\end{remark}
 
 A construction similar to ours can be found in other algebraic structures as well. For instance, for asymmetric product of radical braces \cite{Catino}. Recall that a {\em radical (right) brace} is a set $S$ with two operations $+$ and $\circ$ such that
$(S, +) $ is an abelian group, $(S,\circ)$ is a group and
$(a + b)\circ c + c = a\circ c + b\circ c $ for all $a,b,c\in S$. Also recall that, if $S$ and $T$ are abelian groups, a function $b \colon T\times T\to S$ is a {\em symmetric cocycle} on $T$ with values in $S$ if $
 b(t_1,t_2)= b(t_2,t_1)$, $b(0,0)=0$ and $b(t_1+t_2,t_3)+b(t_1,t_2)=b(t_1,t_2+t_3)+b(t_2,t_3)$
 for every $t_1,t_2,t_3\in T$. This is the analogue of our $k$-bilinear form (notice the vague similarity between the latter condition and Condition (c) in the statement of Theorem~\ref{associative} --- it is necessary to get associativity). If $S$ and $T$ are  radical braces, $b \colon T\times T\to S$ is a symmetric cocycle on $(T,+)$ with values in $(S, +)$,  and $\alpha\colon S\to \Aut(T)$ is a suitable morphism of the multiplicative group $S$ into the group of automorphisms of the radical brace $S$, then the {\em asymmetric product} of $T$ by $S$ \cite[Theorem~3]{Catino} is the cartesian product $S\times T$ with addition and multiplication given by $$( s_1,t_1)+(s_2,t_2)=(s_1+s_2+b(t_1,t_2),t_1+t_2)$$ and $$( s_1,t_1)\circ(s_2,t_2)=(s_1\circ s_2,(\alpha(s_2)(t_1))\circ t_2).$$ Clearly, this addition is related to the topic of this paper.
 
 \bigskip
 
We are grateful to Professor George Janelidze for some suggestions concerning the topic of this paper, and in particular for Example~\ref{Galois} and Remark~\ref{4.1}.

 \end{document}